\documentclass[12pt]{amsart}
\usepackage[T1]{fontenc}
\usepackage{latexsym} 
\usepackage{amsfonts, amsthm, amsmath,amssymb}
\usepackage{times}
\usepackage{color}
\definecolor{refkey}{rgb}{0,0,1}
\definecolor{labelkey}{rgb}{0,0.5,0}
\definecolor{darkgreen}{rgb}{0.0, 0.3, 0.23}
\def\<{{\langle}} 
\def\>{{\rangle}}

\def\C{\mathbb{C}} 

\def\beq{\begin{equation}} 
\def\eeq{\end{equation}}

\def\id{\mathrm{id}}

\newcounter{zlist} 
\newenvironment{zlist}{\begin{list}{(\arabic{zlist})}{ 
\usecounter{zlist}\leftmargin2.5em\labelwidth2em\labelsep0.5em 
\topsep0.6ex%\itemsep0.3ex plus0.2ex minus0.3ex 
\parsep0.3ex plus0.2ex minus0.1ex}}{\end{list}}
\newcounter{blist} 
\newenvironment{blist}{\begin{list}{(\alph{blist})}{ 
\usecounter{blist}\leftmargin2.5em\labelwidth2em\labelsep0.5em 
\topsep0.6ex %\itemsep0.3ex plus0.2ex minus0.3ex 
\parsep0.3ex plus0.2ex minus0.1ex}}{\end{list}} 
\newcounter{rlist} 
 
 \def\stac#1{\raise-.2cm\hbox{$\stackrel{\displaystyle\otimes}{\scriptscriptstyle{#1}}$}}

\def\cten#1{\raise-.2cm\hbox{$\stackrel{\displaystyle\widehat{\otimes}}{\scriptscriptstyle{#1}}$}}

\def\Label#1{\label{#1}\ifmmode\llap{[#1] }\else 
\marginpar{\smash{\hbox{\tiny [#1]}}}\fi} 
\def\Label{\label} 
\newtheorem{proposition}{Proposition}[section]
\newtheorem{lemma}[proposition]{Lemma} 
\newtheorem{corollary}[proposition]{Corollary} 
\newtheorem{theorem}[proposition]{Theorem} 
\theoremstyle{definition} 
\newtheorem{definition}[proposition]{Definition}
 
\newtheorem{example}[proposition]{Example} 
\newtheorem{examples}[proposition]{Examples} 
 
\theoremstyle{remark} 
\newtheorem{remark}[proposition]{Remark} 
 
\numberwithin{equation}{section}

\newcommand{\Cc}{\mathcal{C}}

\def\J{{\bf J}}

\def\*C{{}^*\hspace*{-1pt}{\Cc}}
\def\text#1{{\rm {\rm #1}}}

\def\1{\mathbf{1}}

\newcounter{mnotecount}[section]
\renewcommand{\themnotecount}{\thesection.\arabic{mnotecount}}
\newcommand{\mnote}[1]%{}
{\protect{\stepcounter{mnotecount}}$^{\mbox{\footnotesize
$%\!\!\!\!\!\!\,
\bullet$\themnotecount}}$ \marginpar{%\color{red}%
\raggedright \tiny\em
$\bullet$\themnotecount: #1} }

\begin{document} 

\title{On twisted reality conditions} 

\author[T.\ Brzezi\'nski]{Tomasz Brzezi\'nski}
 \address{ Department of Mathematics, Swansea University, 
  Swansea SA2 8PP, U.K.\ \newline 
\indent Department of Mathematics, University of Bia{\l}ystok, K.\ Cio{\l}kowskiego  1M,
15-245 Bia\-{\l}ys\-tok, Poland} 
  \email{T.Brzezinski@swansea.ac.uk}   

\author[L.\ D\k{a}browski]{Ludwik D\k{a}browski}
\address{SISSA (Scuola Internazionale Superiore di Studi Avanzati), Via Bonomea 265, 34136 Trieste, Italy} 
\email{dabrow@sissa.it}

\author[A.\ Sitarz]{Andrzej Sitarz} 
 \address{Institute of Physics, Jagiellonian University,
prof.\ Stanis\l awa \L ojasiewicza 11, 30-348 Krak\'ow, Poland.\newline\indent
 Institute of Mathematics of the Polish Academy of Sciences,
\'Sniadeckich 8, 00-950 Warszawa, Poland.}
\email{andrzej.sitarz@uj.edu.pl}   

%%%%\date{\today} 
%%%%%%%%%%%%%%%%%%%%%%%%%%%%%%%%%%%%%%%%%%%%%%%%%%%%  
\begin{abstract} 
	
We study the twisted reality condition of Math.\ Phys.\ Anal.\ Geom.\ 19 (2016),
no. 3, Art. 16, for spectral triples, in particular with respect to the product 
and the commutant. Motivated by this 
we present the procedure, which allows one to {\em untwist} the twisted 
spectral  triples studied in Lett.\ Math.\ Phys.\ 106 (2016),  1499--1530. 
We also relate this construction to conformally rescaled real twisted spectral triples, and discuss the untwisting of the 
`minimal twist' procedure of an even spectral triple.
\end{abstract} 
\maketitle 
 
%%%%%%%%%%%%%%%%%%%%%%%%%%%%%%%%%%%%%%%%%%%%%%%%%%%%  
\section{Introduction}

Spectral triples have proven not only to be  a useful tool in index computations
but also to provide a natural setup in which  metric aspects of noncommutative spaces can be understood.  The reality structure and the first-order condition offer a natural framework for
identification of possible Dirac type operators as  noncommutative counterparts 
of first-order differential operators. 
 It is known, however, that even the most fundamental examples of
real spectral triples with the first-order condition do not allow for mild
modifications of the Dirac operator $D$ such as conformal rescaling. A solution to
this problem was proposed in \cite{bcds}. Differently to other approaches cf.\ \cite{CoMo, LaMa} which replace the concept of spectral triples with
twisted spectral triples, 
this proposal does not modify the commutator of $D$ with operators representing the algebra elements
but only the first-order condition and commutation rule of $D$ with the real structure $J$.
It allows for gauge perturbations and contains a large class of examples obtained 
via a conformal modification of the Dirac operator, with the conformal factor built from an element of the algebra and $J$,
so that it belongs to the commutant of the algebra.
This stresses the role of both the algebra and its opposite algebra, which are, in fact, combined into 
the enveloping algebra similarly to what happens for the gauge transformations or for a noncommutative geometric approach to quiver algebras based on double derivations \cite{CraEti:non}.

In this note we study the twisted first-order condition of \cite{bcds} focusing
on the product of spectral triples and the spectral triple for the commutant, 
and relate the construction to real twisted spectral triples of \cite{LaMa}, proving 
that under reasonable conditions the triples can be {\em untwisted}.
In the case of a minimal twist of a standard Dirac operator on a spin manifold this requires a square root of the twisting automorphism, which according to 
\cite{DFLM} can be linked to the Wick rotation in the noncommutative approach to the Standard Model.

The paper is organised as follows. We start by  carefully recalling and clarifying definitions of spectral 
triples with twisted reality \cite{bcds} and real twisted spectral triples \cite{LaMa}. We then proceed to discuss 
twisted spectral triples for a commutant and the restriction imposed on the twist through the tensor product construction.
The main result is contained in Section~\ref{sec.main}, which reveals close relationship between two approaches 
to twisted reality. Finally, in Section~\ref{sec.min} the untwisting procedure is illustrated on an example of a minimal 
twist of \cite{LaMa}.
\section{Twisted reality}
\subsection{Preliminary remarks.}\label{sec.prem}

Let us consider a $\ast$-algebra $A$ with a $\ast$-representation $\pi$ by
bounded operators on a Hilbert space $H$.  The twists are defined by 
using either  an algebra automorphism $\rho$ of $A$, or an automorphism 
$\nu$ of $H$.  Later, in Section~\ref{sec.main} we shall describe 
the precise relation of one to another in case $\rho$ is implemented by $\nu$
through an automorphism of the algebra $B(H)$ of bounded operators 
on $H$.

In order to compare different conventions we introduce the representation 
symbol when considering more than one representation, or more than 
one concrete subalgebra of $B(H)$. In particular, we need to consider two 
real spectral triples \cite{CoMa}, in addition to
$(A, H, \pi, D, J)$ also $(A, H, \pi_J, D, J)$, where
$\pi_J=Ad_J\circ \pi$ is the ($\C$-antilinear) representation anti-unitarily equivalent to $\pi$ via the real structure $J$, that is 
\begin{equation}\label{pi'}
\pi_J(a)=J\pi(a)J^{-1}, \quad \forall a\in A\ .
\end{equation}
The representation $\pi_J$ is closely related to the representation sometimes denoted $\pi^\circ$ of the opposite algebra $A^\mathrm{op}$ ( or, what is essentially the same, to the anti-representation of $A$). Namely, $\pi^\circ=\pi_J\circ \ast$,
where $\ast$ is the star involution on $A$. (Thus, more precisely, we are dealing with the real spectral triple $(A^\mathrm{op}, H, \pi^\circ, D, J)$ rather than $(A, H, \pi_J, D, J)$.)

Spectral triples, possibly twisted, present two aspects: the algebraic aspect and the analytic 
one; these are somehow entangled one with the other.
Concerning the latter, the operator $D$ is required to be self-adjoint on its dense domain Dom$(D)$ and to have compact resolvent. Besides, $\pi(a)$ is bounded 
for any $a\in A$, preserves Dom$(D)$, and has bounded either commutators (in the case of a spectral triple) or twisted 
commutators (in the case of a {\em twisted} spectral triple) with $D$.  Also a reality structure $J$, and the grading 
$\gamma$ (if present) are bounded and preserve Dom$(D)$.

There are also quite natural analytic requirements concerning the twist $\nu$. 
Its domain should contain a dense sub-domain of Dom$(D)$, preserved in 
such a way that all the algebraic identities to be imposed make sense as operators.
We shall not dwell on these aspects here, as they will be checked for the examples 
we will consider, and in the following we focus primarily on algebraic aspects. 

Admittedly these are essentially the only aspects present for finite triples, i.e.\ those with a finite-dimensional $H$. Indeed, in this case the self-adjointness of $D$ becomes basically the  hermicity of its matrix while the other analytic properties are automatically satisfied. Such a finite-dimensional situation appears for instance in the noncommutative approach to the Standard Model of Alain Connes. For a study of twisted reality for $A=\C^2$ and the lowest dimensional $H$ the reader may consult \cite{DabSit:twi}.

To present the formulae more succinctly we introduce the following general bracket.
\begin{definition}\label{def.bracket} 
Let $H$ be a Hilbert space,  $X$ be a set, and let  $\alpha: X\to X$, $\pi: X\to B(H)$ be functions. For any $x\in X$ and  any operator $T$ on $H$  whose domain is preserved by $\pi(x)$,  we  call
\begin{equation}\label{twist.com}
[T,x]_\alpha^\pi:= T\pi(x) - \pi(\alpha(x)) T,
\end{equation}
a  {\em twisted commutator} (or an $\alpha$-twisted $\pi$-commutator, whenever there is a need for keeping track of the functions involved) of $T$ with $x$. 
\end{definition}
We will normally assume that $[T,x]_\alpha^\pi$
 is a bounded operator on a dense subspace of $H$ (the domain of $T$) so that it can be extended to an element of $B(H)$. 
 Needless to say that if $T$ is  a bounded operator on $H$, then $[T,x]_\alpha^\pi$ is defined for all $x\in X$, and it is a bounded operator.
 Let us observe further that if $\alpha$ implements an (algebra) map $\bar{\alpha}: \pi(X)\to B(H)$ in the sense that $\bar{\alpha}(\pi(x)) = \pi(\alpha(x))$, then 
$$
[T,x]_\alpha^\pi = T\pi(x) - \bar\alpha(\pi(x)) T =: [T,\pi(x)]_{\bar\alpha},
$$
where the latter denotes  the usual (algebraic) twisted commutator. Obviously,
$$
[T,x]_{\id}^\pi = [T,\pi(x)].
$$

Twisted commutators satisfy the following version of the Jacobi identity, which can be proven by a direct calculation that uses the associative and distributive laws (which hold in the set-up of Definition~\ref{def.bracket}):
\begin{lemma}\label{lemma.skew-Jacobi}
Let $H$ be a Hilbert space, $X,Y$ sets, and let $\alpha: X\to X$, $\beta: Y\to Y$, $\pi: X\to B(H)$, $\sigma: Y\to B(H)$ be functions. Then, for all operators $T$ on $H$, and all $x\in X$ and $y\in Y$, which satisfy conditions of Definition~\ref{def.bracket}  needed for the definition of twisted commutators,
\begin{equation}\label{lem}
[[T,x]^{\pi}_\alpha ,y]^{\sigma}_\beta - 
[[T,y]^{\sigma}_\beta ,x]^{\pi}_\alpha = 
T\, [\pi(x),\sigma(y)] - [\pi(\alpha(x)) , \sigma(\beta(y))]\,T.
\end{equation}
\end{lemma}
\medskip

\subsection{Spectral triples with twisted reality conditions}
\label{sec.twist}
We begin with the usual background for the construction of spectral triples
that is common for the definition of both usual and
twisted spectral triples. It consists of a $*$-algebra $A$ with
a faithful $*$-representation $\pi:A\to B(H)$ on a Hilbert space $H$ 
and -- in the {\em even} case -- of a grading operator 
$\gamma$ on $H$ that commutes with $\pi(A)$. Furthermore, 
we require the existence of a $\C$-antilinear isometry
$J:H\to H$, such that $J^2 = \epsilon\, \id$ and 
$J\gamma = \epsilon''\gamma J$, where $\epsilon, \epsilon'' \in \{+,-\}$, 
and such that $\hbox{Ad}_J$ takes $\pi(A)$ to its commutant:
\begin{equation}\label{o0c}
[\pi(a),\pi_J(b)] : = [\pi(a), J\pi(b)J^{-1}] =0\quad \forall a,b\in A
 \qquad \mbox{(the zero-order condition)} ,
\end{equation}
where $\pi_J$ is defined in \eqref{pi'}.

Let us also assume that there exists a densely defined, self-adjoint operator $D$ with a compact resolvent, such that the domain of $D$, Dom($D$), is preserved by $\pi(A)$, and -- in the even case -- 
anticommuting with $\gamma$.

The notion of a spectral triple with twisted real structure given below 
is extracted from \cite{bcds}: 
\begin{definition}[{\em Spectral triple with $\nu$-twisted real structure}]
\label{def.reality}
\hfill \ \linebreak[4]
Assume that, for all $a \in A$, $[D,\pi(a)]$ is a bounded operator so  that 
$(A,H, \pi, D)$ is a spectral triple in the sense of \cite{CoMa}.

Let $\nu \in B(H)$ be an invertible bounded operator on $H$ with 
a bounded inverse that implements an algebra automorphism 
$\bar{\nu}$ of $B(H)$, which in turn implements
an algebra automorphism $\hat\nu:A\to A$ of $A$, i.e.,
\begin{equation}\label{proper}
\pi \left( \hat\nu(a)\right) = \nu \pi(a)\nu^{-1} =: \bar{\nu}(\pi(a)).
\end{equation}
Assume that $\nu J$ preserves the domain of $D$ and that
\begin{equation}\label{tc}
DJ\nu = \epsilon' \nu JD,\quad \mathrm{with}\; \epsilon' \in \{+,-\},
\end{equation}
and 
\begin{equation}\label{reg}
\nu J\nu = J.
\end{equation}
By  a {\em $\nu$-twisted first-order condition} we mean that equations
\begin{equation}\label{to1cr}
 [D,\pi(a)] J \pi(\hat{\nu}(b))J^{-1} = J \pi(\hat{\nu}^{-1}(b))J^{-1}[D,\pi(a)],
 \end{equation}
hold for all $a,b\in A$.  

If the triple $(A,H,D)$ is even with grading $\gamma$, then the operator
$\nu$  is requested to satisfy 
$\nu^2\gamma = \gamma\nu^2$. 
The data $(A,H,\pi, D,J,\nu)$ (or $(A,H,\pi, D,\gamma,J,\nu)$ in the case of an even spectral triple) are referred to as {\em spectral triple with twisted real structure} or {\em spectral triple with $\nu$-twisted real structure}.
$\diamond$
\end{definition}

The notion of a real twisted spectral triple is given in the following 
restatement of definitions in \cite{LaMa}: 

\begin{definition}[{\em Real twisted spectral triple}]
	\label{def:realtwist}
	\hfill \ \linebreak[4]
	Let $\rho$ be an automorphism of $A$ 
	such that 
	\begin{equation}\label{star-1}
	\rho\circ * = *\circ \rho^{-1}.
	\end{equation}
	Assume that, 
	for all $a \in A$, the twisted commutators $[D, a]^\pi_\rho$ are bounded (so that $(A,H,D)$ is a {\em $\rho$-twisted spectral triple} in the sense 
	of \cite{CoMo}), and that 
	\begin{equation}\label{c}
	DJ = \epsilon' JD\quad \mathrm{with}\; \epsilon' \in \{+,-\}.
	\end{equation} 
	
	By a {\em twisted first-order condition} we mean that equations  
	\begin{equation} \label{eq:39-bis}
	[\left[D,a\right]_\rho^\pi, b ]_\rho^{\pi_J} =0
	\end{equation}
	hold for all $a, b \in A$. 
The data $(A,H,\pi, D,\rho, J)$ (or $(A,H,\pi, D,\rho,\gamma, J)$  in the case of an even twisted spectral triple) are referred to as {\em real twisted spectral triple} or {\em real $\rho$-twisted spectral triple}. 	
	$\diamond$
\end{definition}

We note that $\hat\nu$ of Definition~\ref{def.reality} satisfies
\begin{equation}\label{nu*}
\widehat{\nu*}\circ * \circ \hat\nu\circ *   = 1 .
\end{equation}
Thus although originally the compatibility with the star structure was not explicitly assumed in \cite{bcds}, 
when $\nu = \nu^*$ (which was satisfied in the examples considered therein) then it follows that	
$$\hat\nu\circ * = *\circ \hat\nu^{-1},$$
in concord with \eqref{star-1}. 

Examples of twisted reality conditions include conformal rescalings 
of the Dirac operator. 

\begin{examples}\label{exlama}
An example of the spectral triple with $\nu$-twisted real structure was 
constructed in \cite{bcds} using the conformal twist. More precisely, 
let $(A,H,\pi, D, J)$ be a real spectral triple. For all invertible $u \in A$ such that $k:=\pi(u)$ is 
positive and with the bounded inverse $k^{-1}$, set $k':= JkJ^{-1}$ to be the element of the commutant 
of $\pi(A)$. Then, with 
$$
D_k := k'Dk', \qquad \nu:= k^{-1}k',
$$
$(A,H,\pi, D_k,J,\nu)$ satisfies \eqref{tc}, 
\eqref{reg} and the $\nu$-twisted first-order condition \eqref{to1cr}. 
On the other hand,
$$
\tilde{D}_k := k'kDk'k, \qquad \rho: A\to A, \quad  a 
\mapsto u^{2} a u^{-2},
$$ 
leads to a twisted real spectral triple $(A,H,\pi, \tilde{D}_k, J, \rho)$ 
satisfying the twisted first-order condition \eqref{eq:39-bis}.
\;\; {$\diamond$}
\end{examples}

Another example of the twisted real spectral triple arises from the minimal twist of an even spectral triple, 
for example of the classical spectral triple on an even dimensional spin manifold; see \cite{LaMa}. We shall 
discuss this example in detail in Section~\ref{sec.min}.

\section{Twisted spectral triple for the commutant and tensor product of spectral triples with twisted reality}

\subsection{Twisted spectral triple for the commutant}
In the case of an ordinary real spectral triple $(A,  H, \pi, D, J)$
the relation of the reality operator $J$ with
the Dirac operator $D$ assures that $(A,  H, \pi_J, D, J)$ or, more precisely, $(A^\mathrm{op}, H, \pi^\circ, D, J)$
is also a real spectral triple, and if the former one obeys
the first-order condition so does the latter one.
In this section we would like to investigate what happens in a more general case of spectral 
triples with a twisted real structure. 

Lemma~\ref{lemma.skew-Jacobi} shows an interesting exchange of properties between the twisted 
commutators and order-one condition. In fact we can state the following
\begin{proposition}\label{prop.exchange}

In the setup of Definition~\ref{def.reality}, the $\nu$-twisted first-order condition 
\eqref{to1cr} is equivalent to 
\begin{equation}\label{to1cr.ex}
 [[D, {\pi_J(b)}]_{\bar{\nu}^{2}},\pi(a)] = [[D, b]^{\pi_J}_{\hat{\nu}^{-2}},\pi(a)]  = 0,  \qquad \forall a,b\in A.
\end{equation}
\end{proposition}
\begin{proof}
This is a consequence of Lemma~\ref{lemma.skew-Jacobi} and the zero-order and regularity 
conditions \eqref{o0c}, \eqref{reg}. First note that, taking into account the regularity 
condition \eqref{reg} one easily finds that
\begin{equation}\label{nu-nu}
 [D,\pi_J(b)]_{\bar{\nu}^2} = [D, b]^{\pi_J}_{\hat{\nu}^{-2}},
\end{equation}
where $\bar\nu$ is the algebra automorphism induced by $\nu$ \eqref{proper}. This establishes 
the first equality in \eqref{to1cr.ex}. 

Next observe that, owing to the fact that $\hat\nu$ is an automorphism, the  condition \eqref{to1cr} 
can be succinctly written in terms of twisted commutators as
\begin{equation}\label{twist.com.1o}
[[D,\pi(a)] ,b]^{\pi_J}_{\hat{\nu}^{-2}} = [[D,a]^\pi_\id\, , b]^{\pi_J}_{\hat{\nu}^{-2}} =0, \qquad \forall a,b\in A.
\end{equation}
Further, choose
  $X=Y =A$, $T=D$, $x=a$, $y=b$, $\pi$ equal to the representation $\pi: A \to B(H)$, and $\sigma = \pi_J$ in Lemma~\ref{lemma.skew-Jacobi}. Irrespective of the choices of $\alpha$ and $\beta$ (not yet made), the zero-order condition \eqref{o0c} implies that the right hand side of \eqref{lem} vanishes, so that
\begin{equation}\label{alpha.beta}
[[D,a]^{\pi}_\alpha ,b]^{\pi_J}_\beta = [[D,b]^{\pi_J}_\beta ,a]^{\pi}_\alpha.
\end{equation}
Finally, choosing 
$\alpha = \id$ and $\beta = \hat{\nu}^{-2}$ we obtain the equality of the middle term in \eqref{twist.com.1o} with the middle term in \eqref{to1cr.ex}. Hence \eqref{to1cr.ex} is equivalent to the $\nu$-twisted first-order condition written in the commutator form \eqref{twist.com.1o}.
\end{proof}

\begin{corollary}\label{cor.comm}
If $(A,H, \pi, D,J,\nu)$ is a spectral triple with $\nu$-twisted first-order
condition, then $(A^\mathrm{op}, H, \pi^\circ, D, \hat{\nu}^{-2})$ is a $\hat{\nu}^{-2}$-twisted spectral 
triple (with bounded twisted commutators) for which $J$ satisfies the usual (i.e.\ with the untwisted outer 
commutator) first-order condition. 
\end{corollary}

{
\begin{example}\label{como.twist}
The twisted spectral triple obtained 
from a real spectral triple $(A, H, \pi, D, J)$ in \cite{CoMo}  is an easy example illustrating the transition described in Corollary~\ref{cor.comm}.  First note that since $(A^\mathrm{op}, H, \pi^\circ, D, J )$ 
is also a real spectral triple and, therefore, following the first of Examples~\ref{exlama} we can construct
a spectral triple with $\nu$-twisted reality condition. Let us fix a positive element $u \in A$ and set
$k = \pi(u)$ and $k' = J k J^{-1}$. Since $u=u^*$,  $k' = \pi^\circ(u)$ and by taking $\nu = k (k')^{-1}$
similarly to Example \ref{exlama} we obtain a spectral triple 
$(A^\mathrm{op}, H, \pi^\circ, D_k, J , \nu)$ with a $\nu$-twisted real structure 
for $D_k = k D k$.

The Corollary \ref{cor.comm} demonstrates in turn that $(A, H, \pi, D_k, J)$, is a twisted spectral 
triple (with the twist implemented by the algebra automorphism $\rho(a) = \hat{\nu}^{2}(a)$) and the real structure that 
satisfies the untwisted first-order condition. Indeed, since $(\pi^\circ)_J(a) = \pi(a^*)$ 
taking \eqref{to1cr.ex}, for any $a,b \in A$, we obtain:
$$ [[D_k, \pi(b^*)]_{\bar{\nu}^{2}},\pi_J(a^*)] = 0,  
\qquad \forall a,b\in A.$$
Since $a,b$ are arbitrary we conclude that $(A,H,\pi,D_k)$ is a twisted spectral triple,
where the twisting automorphism is $\rho(a)=\hat{\nu}^{2}(a) =  u^2 a u^{-2}$,
whereas the real structure (which by construction satisfies $\nu J \nu =J$, $D_k J\nu = \epsilon' \nu JD_k$)
satisfies the non-twisted first order condition.

This gives the correct notion of a real structure and the first-order condition for a twisted spectral triple 
first introduced by \cite{CoMo}, with the Dirac operator conformally rescaled by an element from the 
algebra and the conformal twisting automorphism. Moreover, the example demonstrates duality between 
a spectral triple for an algebra (with the $\nu$-twisted first order condition) and a twisted spectral 
triple for its commutant (and the non-twisted first-order condition) with an appropriate relations between 
the twist, reality structure and the Dirac operator as in \cite{bcds}.
{$\;\; \diamond$}
\end{example}
}

In contrast, note that in the setup of Definition~\ref{def:realtwist}, setting 
$\alpha = \beta =\rho$ in \eqref{alpha.beta} we obtain that the twisted first-order condition 
for the twisted real spectra triples \ref{def:realtwist} is equivalent to 
\begin{equation}\label{eq:39-bis.ex}
[[D, b ]_\rho^{\pi_J},a]_\rho^\pi =0, \qquad \forall a,b\in A.
\end{equation}
In view of the fact that $(\pi_J)_J=\pi$,  this establishes an equivalence of 
the twisted first-order condition \eqref{eq:39-bis} for $\pi(A)$ with identically 
twisted condition for $\pi_J(A)$. In fact, this has been already observed in 
\cite[Equation~(2.13)]{LaMa} (in the representation-free notation adopted there) 
as an equivalence of 
$[ [D,a]_\rho, Jb^*J^{-1}]_{\rho^\circ}=0 $
with
$[[ D, Jb^*J^{-1}]_{\rho^\circ},a]_{\rho}=0.$

\subsection{Tensor product of spectral triples with twisted reality}
\label{sec:tens}

Although tensor products of spectral triples and real spectral triples are
well defined, the situation changes drastically in the case of twisted 
reality conditions and twisted commutators. The reason for that is the
fact that the Dirac operator for the tensor product is given as the sum 
$D_1 \otimes \gamma_2+ \id \otimes D_2$ and, therefore, the commutator with
any simple tensor is also equal to the sum $[D_1, a_1] \otimes \gamma_2 a_2
+ a_1 \otimes [D_2, a_2]$. Since the $\nu$-twisted first-order condition
is different from the commutant condition (or the zero-order condition) one 
cannot recover the twisted first-order condition for the product for general 
twists. 

However, since the twist automorphism in the $\nu$-twisted first-order 
condition (\ref{to1cr}) appears effectively as $\hat{\nu}^2$, the tensor
product of two spectral triples with twisted reality is possible if
$\hat{\nu}^2 = \id$. 

\begin{lemma}
Let $(A_1, H_1,\pi_1,D_1, \gamma_1, J_1,\nu_1)$ and 
$(A_2,  H_2, \pi_2, D_2, \gamma_2, J_2, \nu_2)$ be two 
even real spectral triples that satisfy the twisted reality conditions as defined 
in Definition~\ref{def.reality} with the respective KO-dimensions given by the
signs $\epsilon_i,\epsilon_i',\epsilon''_i$, $i=1,2$. Additionally, let us 
assume that $\nu_i^2 = \alpha_i$ and 
$\nu_i \gamma_i = \beta_i \gamma_i \nu_i$, with $\alpha_i,\beta_i$ 
being $\pm 1$.

Then the usual tensor product of  spectral triples, with 
$(A_1 \otimes A_2,  \pi_1 \otimes \pi_2, 
H_1 \otimes H_2, D_1 \otimes \id + \gamma_1 \otimes D_2, 
J_1 \otimes J_2, \gamma_1 \otimes \gamma_2, \nu_1 \otimes \nu_2)$ 
is a spectral triple that satisfies the  $\nu_1 \otimes \nu_2$-twisted first-order condition provided that
$$ \alpha_2 \epsilon_2' = \alpha_1 \epsilon_1' \beta_1 \epsilon_1''.$$
\end{lemma}

As the proof is straightforward we omit it, remarking only that
the situation with other possible choices for the reality 
structure $J$ (for details see \cite{DaDo}) can be treated in 
a similar way.

\begin{example}\label{tensex}
Let $(A_1,H_1,\pi_1, D_1,J_1,\gamma_1)$ be an even spectral triple of 
KO-dimension $0$ and let us consider an even spectral 
spectral triple over the algebra of functions on two points, 
$\C \oplus \C$, with the smallest faithful representation on $\C^2$. 

As already argued in \cite{DabSit:twi}, there is no nondegenerate real triple 
with such a representation as the first-order condition would not be satisfied,
so the only possibility is to fix $D_2=0$. Nevertheless, the real 
structure (complex conjugation) as well as the twist (the nontrivial
permutation on the Hilbert space $\C^2$) and the usual grading 
define the following spectral triple with $\nu$-twisted real structure,
$$ A_2 = \C \oplus \C, \; H_2 = \C \oplus \C, \; D_2=0, \;
J_2 = \hbox{id} \circ *, \; \gamma_2 = \hbox{diag}(1,-1), \; 
 \nu_2 = \left( \begin{array}{cc} 0 &1 \\ 1 & 0 \end{array} \right). $$

The tensor product of these two spectral triples is still a spectral triple, which is 
degenerate in the sense that the kernel of the commutator with $D$ is bigger than $\C$, 
yet it has an $\hbox{id}\otimes \nu_2$-twisted real structure in the sense of Definition~\ref{def.reality}. 

The representation, Dirac operator and the twist are implemented on the Hilbert space as:
$$ \pi(f,g) =    \left( \begin{array}{cccc} 
	f & 0 & 0 & 0 \\
	0 & g & 0 & 0 \\
	0 & 0 & f & 0 \\
	0 & 0 & 0 & g \\
\end{array} \right),  
\quad
D =    \left( \begin{array}{cccc} 

0 & 0 & \partial_+ & 0  \\
0  & 0 & 0 & \partial_+\\
\partial_-  & 0 & 0 &  0 \\
0 & \partial_- & 0 & 0 
\end{array} \right),  
\quad
\nu =    \left( \begin{array}{cccc} 
0 & 1 & 0 & 0 \\	
1 & 0 & 0 & 0 \\
0 & 0 & 0 & 1 \\
0 & 0 & 1 & 0 
\end{array} \right),  
$$
with the reality structure $J$ and grading $\gamma$:
$$ 
J =    
\left( \begin{array}{cccc} 
	J & 0 & 0 & 0 \\
	0 & J & 0 & 0 \\
	0 & 0 & J & 0 \\
	0 & 0 & 0 & J \\
\end{array} \right),  
\qquad
\gamma =    \left( \begin{array}{cccc} 
	1 & 0 & 0 & 0 \\
	0 & -1 & 0 & 0 \\
	0 & 0 & -1 & 0 \\
	0 & 0 & 0 & 1 \\
\end{array} \right).  
$$ 
Later on, we will see how the above product construction can be identified 
with the spectral triple arising from the {\em untwisted} minimal twist. \quad $\diamond$
\end{example}

\section{The untwisting of real twisted spectral triples}
\label{sec.main}

The main result of this note, i.e.\ the explanation of the relation between twisted real 
spectral triples \cite{LaMa} and the spectral triples with $\nu$-twisted real structure, in particular, 
the relation between the Dirac operators, is contained in the following theorem.

\begin{theorem}\label{thm.main}
Let $A$ be a $*$-algebra and $\tilde{\pi}: A\to B(H)$ be a $*$-representation  of $A$ on a Hilbert space $H$. Let $J:H\to H$ be a $\C$-antilinear isometry such that $J^2 = \epsilon$ and that the zero order condition \eqref{o0c} is satisfied. Let $\rho$ be an algebra automorphism satisfying \eqref{star-1}, and  let $\nu$ be a bounded operator on $H$ with the bounded inverse such that
\begin{blist}
\item $\nu$ implements an algebra automorphism $\hat{\nu}$ of $A$ in representation $\tilde{\pi}$ as in \eqref{proper} and  $\rho = \hat{\nu}^{-2}$,
or
\item $\nu$ is a unitary operator such that $\nu^{-2}$ implements  $\rho$  in representation $\tilde{\pi}$ as in \eqref{proper}. 
\end{blist}
Let
\begin{equation}\label{pi.nu}
\pi_\nu : A \to B(H), \qquad a\mapsto \nu^{-1} \tilde{\pi}(a) \nu,
\end{equation}
be the induced representation of $A$, and set
\begin{equation}\label{tilde.pi}
{\pi} = \begin{cases}  \tilde{\pi}, & \mbox{in case (a)},\cr
\pi_\nu, & \mbox{in case (b)},
\end{cases}
\end{equation} 
so that ${\pi}$ is always a $*$-representation. Assume further that 
\begin{equation}\label{rho-nu}
\nu J \nu = J .
\end{equation}
For an operator $\tilde{D}$ on $H$, set
\begin{equation}\label{dtilded}
D = \nu \tilde{D} \nu ,
\end{equation}
Then:
\begin{zlist}
\item $({\pi}, D, J, \nu^2)$ satisfy conditions \eqref{proper}--\eqref{to1cr} of a spectral triple with a $\nu^2$-twisted 
real structure if and only if $(\tilde{\pi},\tilde{D}, J, \rho)$  satisfy
conditions \eqref{c}--\eqref{eq:39-bis} of real $\rho$-twisted spectral triple. \item 
If $(A,  H,\tilde{\pi}, \tilde{D}, ,\rho, J)$ is a real $\rho$-twisted spectral triple, $\nu$ is self-adjoint and there 
exists $w \in \C$ outside of the spectrum of $\tilde{D}$ and $z\in \C$ such that
\begin{equation}\label{Kato}
\|w-z\nu^{-2}\| \|(\tilde{D} - w)^{-1}\| < 1, 
\end{equation}
then $(A, H, \pi, D, J, \nu^2)$  is a spectral triple with $\nu^2$-twisted real structure. 
\end{zlist}
\end{theorem}

\begin{proof}
First note that, for any operators $x$, $y$ on $H$ for which the forthcoming expressions can be defined (indeed for any elements 
$x,y,\nu$ of any associative algebra, with $\nu$ invertible),
$$
 xy - \nu^{-2} x \nu^2 y = \nu^{-1} [ \nu x \nu, \nu^{-1}y \nu ] \nu^{-1}. 
$$
In particular, since, by either of the assumptions (a) or (b), for all $a\in A$, 
\begin{equation}\label{rho.nu.nu}
\tilde{\pi}(\rho(a)) = 
\nu^{-2}\tilde{\pi}(a)\nu^2,
\end{equation} 
and $D = \nu\tilde{D}\nu$,
\begin{equation}\label{com-twicom}
[D, \pi_\nu(a)] = \nu\, [ \tilde{D}, a ]_\rho^{\tilde{\pi}}\,\nu ,
\end{equation}
and so since $\nu$ and $\nu^{-1}$ are bounded, the commutators of $D$ with $\pi_\nu(a)$
are bounded whenever the twisted commutators 
 $[\tilde{D},a]^{\tilde{\pi}}_\rho$  are bounded.
Moreover 
$$
\begin{aligned}
\, [ \tilde{D}, a ]_\rho^{\tilde{\pi}} \, J {\tilde{\pi}}(b) J^{-1} &= \nu^{-1} [D, \pi_\nu(a))] \, 
\nu^{-1} J {\tilde{\pi}}(b) J^{-1} \nu \nu^{-1} \\
&= \nu^{-1} \left( [D, \pi_\nu(a)] J \pi_\nu(\rho^{-1}(b)) J^{-1} \right) \nu^{-1},
\end{aligned}
$$
 by the definition \eqref{pi.nu} of $\pi_\nu$,
\eqref{rho-nu} and by \eqref{rho.nu.nu}.
On the other hand:
$$
\begin{aligned}
{\tilde{\pi}}_J(\rho(b)) [ \tilde{D}, a ]_\rho^{\tilde{\pi}}  
& = \nu^{-1} \nu J \nu^{-2} {\tilde{\pi}}(b) \, \nu^2 J^{-1} \nu^{-1} [D, \pi_\nu(a)] \nu^{-1} \\
& = \nu^{-1} \left( J (\pi_\nu(\rho(b))) J^{-1} [D, \pi_\nu(a)] \right) \nu^{-1},
\end{aligned}
$$
where again the definition \eqref{pi.nu} of $\pi_\nu$,  
\eqref{rho-nu} and \eqref{rho.nu.nu} were used.
Therefore,  $\tilde{D}$ satisfies the twisted first-order condition \eqref{eq:39-bis} if and only if
the operator $D$ satisfies:
\begin{equation}\label{conclusion}
 [D,\pi_\nu(a)] J \pi_\nu(\rho^{-1}(b)) J^{-1} = J (\pi_\nu(\rho(b))) 
J^{-1} [D,\pi_\nu(a)].
\end{equation}
In the case (a), i.e.\ when $\nu$ implements an automorphism $\hat{\nu}$ of $A$, $\pi_\nu = {\tilde{\pi}} \circ \hat{\nu}^{-1}$, 
we can replace $a$ and $b$ in \eqref{conclusion} by their images under $\hat{\nu}$ and thus, bearing in mind that $\rho$ is 
implemented by $\nu^{-2}$, $D$ satisfies the $\nu^2$-twisted first-order condition in representation ${\tilde{\pi}}$. 
In the case (b) \eqref{conclusion} gives this condition in representation $\pi_\nu$.

Moreover, if $J \tilde{D} = \epsilon' \tilde{D} J$ then:
$$ DJ \nu^2 = \nu \tilde{D} \nu J \nu^2 = \nu \tilde{D} J \nu =
\epsilon' \nu J \tilde{D} \nu = \epsilon' \nu^2 J \nu \tilde{D} \nu = \epsilon' \nu^2 JD,$$
where we used  \eqref{rho-nu}. By the same token $DJ \nu^2 = \epsilon' \nu^2 JD$ implies 
that $J \tilde{D} = \epsilon' \tilde{D} J$.
The identity $\nu^2 J \nu^2 = J$ follows trivially from  \eqref{rho-nu}. This establishes  
one-to-one correspondence stated in the first assertion. 

Now assume that $\tilde{D}$ has a compact resolvent  and that there exist $w,z\in \C$,
with $w$ not in the spectrum of $\tilde{D}$, such that the inequality \eqref{Kato} holds. Then 
the inverse of
$T= \tilde{D} - w$  exists and is compact, the operator $A= w-z\nu^{-2}$ is bounded and hence
$$
\tilde{D}- z\nu^{-2} = T +A,
$$
has the compact inverse by \cite[Chapter IV, Theorem~1.16 \& Remark~1.17]{Kato:book}. 
Therefore, since
$$ (D-z)^{-1}= ( \nu \tilde{D}\nu -z)^{-1}=
\nu^{-1} (\tilde{D}- z\nu^{-2})^{-1}  \nu^{-1}, $$
and $\nu^{-1}$ is bounded we see that $D$ has also compact resolvent. 

If, in addition to other assumptions, $\nu$ and $\tilde{D}$ are self-adjoint, then the self-adjointness of $D$ is clear, which, in combination with the first assertion, establishes the second one.
\end{proof}

Note that if $\tilde{D}$ is invertible, then we can choose $w=z=0$ to satisfy the inequality \eqref{Kato}. In this case it is also clear that $D$ is invertible. 

As far as the the grading is concerned, we remark that in the setup of Theorem~\ref{thm.main},  if $(A, H,{\tilde{\pi}},  \tilde{D}, \rho, \gamma, J)$ 
is an even real twisted spectral real triple with grading $\gamma$ in the sense of Definition~\ref{def:realtwist}, then  
$(A, H, {\pi}, {D}, \nu^{-1}\gamma\nu, J, \nu^2)$ is an even spectral triple with $\nu^2$-twisted real structure in the 
sense of Definition~\ref{def.reality}, provided $\nu^2\gamma = \gamma \nu^2$.

\begin{example}\label{kkprime}
Consider a real spectral triple $(A, H,\pi, D_0, J)$ and  fix  $u \in A$, 
$k = \pi(u)$ a positive element, $k' = J k J^{-1}$. Then, from Examples~\ref{exlama} we know
that  $\tilde{D}_k = k k' D_0 k' k$ is a Dirac operator for a real twisted 
spectral triple with twisted first-order condition, for the inner algebra automorphism $\rho(a) = u^2 a u^{-2}$. Applying the untwisting procedure from 
Theorem~\ref{thm.main} we obtain $\nu = k' k^{-1}$ and, therefore,
$D = \nu \tilde{D} \nu = (k')^2 D_0 (k')^2$  is a Dirac operator for a spectral 
triple as in Definition~\ref{def.reality}. Observe that $D$ is indeed an operator that 
arises from $D_0$ by a conformal twist with the real structure twisted by $\nu^2 = (k')^{2} k^{-2}$ in
total agreement with  the statements of Theorem~\ref{thm.main}. \quad $\diamond$
\end{example}

If the assumption that $\nu$ is a self-adjoint operator  is relaxed, there is no guarantee that the operator $D$ obtained from $\nu\tilde{D}\nu$ be self-adjoint. However, even if $\nu$ is not a self-adjoint operator, the genuine (even) spectral triple with 
$\nu^2$-twisted real structure can be obtained from  $(A, H, \pi, \tilde{D}, J)$ by the {\em doubling  procedure} which we 
outline presently.

In the setup of Theorem~\ref{thm.main}, since $\rho$ satisfies condition \eqref{star-1}, the operator on $H$ implementing it should be self-adjoint, so it is natural to expect that $\nu^{*2} = \nu^2$. By $*$-conjugating conditions \eqref{tc}, \eqref{reg} and \eqref{to1cr} if $D$ satisfies these conditions for  a real structure $J$ and twisting $\nu^2$, then $D^*$ also satisfies them for the real structure $\pm J^{*}$ and twisting $\nu^{*2}$. Since $\nu^2$ is self-adjoint, choosing $J$ to be (anti-)self-adjoint we conclude that both $D$ and  $D^*$ can be equipped with the same $\nu^2$-twisted real structure. We then can consider a new representation of $A$ (still denoted by $\tilde{\pi}$), which doubles the representation space to $H\oplus H$ and define
$$
\mathbf{D} = \begin{pmatrix} 0 & D \cr D^* & 0\end{pmatrix}, \qquad {\pmb{\nu}} = \begin{pmatrix} \nu & 0 \cr 0 & \nu \end{pmatrix}, \qquad {\pmb{\gamma}} = \begin{pmatrix} 1 & 0 \cr 0 & -1\end{pmatrix},
$$
and depending on the required $KO$-dimension, choose the real structure $\mathbf{J}$ to be either a diagonal or an off-diagonal $2\times 2$-matrix with operator entries $J$. Since $\mathbf{D}$ is now evidently a self-adjoint operator, $(A, \tilde{\pi}, H\oplus H ,\mathbf{D}, {\pmb{\gamma}}, \mathbf{J}, {\pmb{\nu}}^2)$ is an even spectral triple with the ${\pmb{\nu}}^2$-twisted real structure.

\section{Untwisting the minimal twist}\label{sec.min}
The minimal twist of a Dirac operator on an even dimensional manifold is 
an example illustrating the notion of a real twisted spectral triple with a 
simple flip automorphism described in \cite{LaMa}. In this section we 
briefly sketch this example and then use Theorem~\ref{thm.main} to transfer 
it to an example of a spectral triple with a $\nu$-twisted real structure. 
We demonstrate that by passing to such a picture  we can interpret this construction 
as a shadow of a spectral triple constructed as a tensor product.

Consider an even dimensional spin manifold $M$ and a usual Dirac operator
$\tilde{D}$ on the Hilbert space of square integrable sections of the spinor
bundle, $H = H_+\oplus H_-$, with the representation $\tilde{\pi}_0$ of $C^\infty(M)$,
and the grading $\gamma$:
\begin{equation}
\tilde{D} = \left( \begin{array}{cc} 0 & \partial_- 
\\ \partial_+ & 0 \end{array} \right),  \qquad
\tilde{\pi}_0(f) = 	  \left( \begin{array}{cc} f & 0  
\\ 0 & f \end{array} \right),  \qquad
\gamma = \left( \begin{array}{cc} 1 & 0\\
0& -1 \end{array} \right).
\label{Dpg}
\end{equation} 
where $\partial_+ = \partial_-^*$.

Using the grading $\gamma$ one can extend the representation of the 
algebra $C^\infty(M)$ to the algebra 
$A = C^\infty(M)\oplus C^\infty(M)\simeq C^\infty(M) \otimes \C^2$ 
on the same Hilbert space through the following construction:
\begin{equation}\label{reps} 
\tilde{\pi}(f , g) = \frac{1}{2} \tilde{\pi}_0(f) (1+\gamma)   - \frac{1}{2} \tilde{\pi}_0(g) (1-\gamma) 
= \hbox{diag}(f,g), 
\qquad \mbox{where} \quad f,g\in C^\infty(M)\,.
\end{equation}
Observe that extending the algebra we loose the chirality $\gamma$,
which becomes now an algebra element and not an external grading of 
the Hilbert space, so technically we pass from an even to odd spectral
triple. The `flip' algebra automorphism
\begin{equation}\label{flip}
\rho :A\to A, \qquad (f,g) \mapsto (g,f),
\end{equation}
is  implemented on $H$ by:
$$ \check\rho = \left( \begin{array}{cc} 0 & 1 	
\\ 1& 0 \end{array} \right), $$
which acts on $ \pi(f,g)$ in a way clearly reflecting the definition \eqref{flip} of $\rho$:
$$ \bar{\rho} \left( \hbox{diag}( f,g) \right) = \hbox{diag}(g,f). $$

Note that $\bar\rho$ does not commute with the Dirac operator 
$\tilde{D}$. Following \cite{LaMa} we thus obtain a real $\tilde{\rho}$-twisted spectral triple 
and the twisted commutators of $\tilde{D}$ with the elements of the algebra represented through $\tilde\pi$ come out as

$$
[\tilde{D}, (f, g)]_\rho^{\tilde{\pi}} =  [\tilde{D},\hbox{diag}(f,g) \, ]_{\bar\rho} 
= \begin{pmatrix} 0 & [\partial_-, f] 	\\ 
[\partial_+, g] & 0 \end{pmatrix}.
$$

Moreover, if $J \gamma = \epsilon'' \gamma J$  then any
element of the commutant is of the form $\hbox{diag}(f^\circ, g^\circ)$,
and it is easy to verify that the twisted version of the
first-order condition holds, in concord with  Proposition~\ref{prop.exchange}.

Next, in order to pass to an ordinary spectral triple we need to use a square
root $\nu$ of the operator $\check\rho$, e.g.,
\begin{equation}\label{numin}
\nu = \frac{1}{2} \left( \begin{array}{cc} (1+i) & (1-i) 	
\\ (1-i) & (1+i) \end{array} \right) = \frac{1}{2} \left( (1+i) + (1-i) \check\rho \right). 
\end{equation}
At this point it is worth noting that $\nu$ is not self-adjoint, but it is unitary instead. 
Irrespective of this, the construction presented in the first (algebraic) part of Theorem~\ref{thm.main} can be performed. 
We will discuss later on how to deal with non-self-adjointness and other analytic aspects of $D$.

The representation $\pi_\nu$ comes out as

$$ \pi_\nu(f,g) = \frac{1}{2} 
\left( \begin{array}{cc} (f+g) & i (g-f) 	
\\ -i (f-g) & (f+g) \end{array} \right), $$
and the $\nu$-transformed Dirac operator reads:
$$ 
D = \nu \tilde{D} \nu = 
\frac{1}{2} \left( \begin{array}{cc} (\partial_+ + \partial_-) & -i (\partial_+ - \partial_-)
\\ i (\partial_+ - \partial_-)& (\partial_+ + \partial_-) \end{array} \right)
= \left( \begin{array}{cc} \partial_1 &  \partial_2
\\  - \partial_2 &  \partial_1 \end{array} \right) , 
$$
where we have introduced the  following notation for self-adjoint operators,
$$ \partial_1 = \frac{1}{2} (\partial_+ + \partial_-), \qquad 
\partial_2 =-\frac{i}{2} (\partial_+ - \partial_-). $$
The commutator becomes:
$$
[D, \pi_\nu(f,g)] = \frac{1}{2} \left( \begin{array}{cc} 
[ \partial_+, f] + [ \partial_-, g]  &
i ([ \partial_-, g] - [ \partial_+, f] ) \\
i ([ \partial_+, f] - [ \partial_-, g]) 
&[ \partial_+, f] - [ \partial_-, g] 
\end{array} \right). 
$$

Next, we verify the relation between $J$ and $\nu$. Following \cite{LaMa} we consider 
two cases depending on the commutation between $J$ and $\gamma$
of \eqref{Dpg}. 
If $\gamma$ commutes with $J$, then $J$ is necessarily diagonal,
$$ J =    \left( \begin{array}{cc} J_+ & 0 	
\\ 0 & J_- \end{array} \right), $$
and then

$$ 
\begin{aligned}
\nu J \nu  
= \frac{1}{2} \left( \begin{array}{cc} J_+ +J_-  & i (J_+ - J_-) 	
\\ -i (J_+ - J_-) & J_+ +J_-  \end{array} \right),  
\end{aligned}$$
so $\nu J \nu = J$ if and only if $J_+ = J_-$. 
Similarly, if $\gamma$ anti-commutes with $J$, then $J$ is off-diagonal,
$$ J =    \left( \begin{array}{cc} 0 & J_- 	
\\ J_+ & 0  \end{array} \right), $$
and
$$
\begin{aligned}
\nu J \nu 
= \frac{1}{2} \left( \begin{array}{cc} -i (J_+ -J_-)  & J_+ +J_-
\\J_+ +J_- &  i (J_+ - J_-)    \end{array} \right). 
\end{aligned}
$$
Hence again $J_+=J_-$ is necessary and sufficient for the
compatibility with $\nu$. Thus we conclude that $J$ must be 
a suitable multiple of either the identity matrix or $\check\rho$.

The presentation of the spectral triple with the twisted reality
structure takes simpler form when we pass back to the diagonal
representation by conjugating all operators by $\nu$.

\begin{remark}\label{rem51}
	The triple $(A, H, \pi_\nu, D,J)$ (with twisted real structure $J$) is unitarily equivalent to the triple 
	$(A, H, \pi, D', J')$, where  $\pi$ is the diagonal representation \eqref{reps},
	$$ D' =  \check\rho \tilde{D} = 
	\left( \begin{array}{cc} \partial_+ & 0 	
	\\ 0 & \partial_- \end{array} \right), 
	\qquad 
	J' =   \check\rho J,
	$$
	and $\check\rho = \nu^2 = \nu^{-2}$. $\diamond$
\end{remark}

\begin{remark}
Since $\nu$ is not self-adjoint operator $D$ is not a self-adjoint either.
Thus it does not represent a valid Fredholm module.
Nevertheless due to the unitarity of $\nu$ and the anticommutation 
rule of $\tilde{D}$ with $\gamma$, ${D}$ is at least normal, i.e. $D$ and $D^*$ commute one with another, and this permits one to 
develop a continuous calculus. $\diamond$
\end{remark}  
Furthermore we would like to advocate the point of view that the presented 
construction is a {\em shadow} of a spectral triple as constructed 
by the doubling procedure presented in Section~\ref{sec.main}, which we apply 
now. For simplicity we consider only the case of $J\gamma = \gamma J$.

We claim that the doubled untwisted spectral triple of the minimal twist as discussed
in  remark \ref{rem51} is unitarily equivalent to the spectral triple of the
tensor product from Example \ref{tensex}. First, applying the doubling procedure
we obtain the spectral triple with $\check\rho$-twisted real structure. Explicit 
representation, Dirac operator, grading, reality and the twist $\check\rho$ are:
$$ \pi(f,g) =    \left( \begin{array}{cccc} 
f & 0 & 0 & 0 \\
0 & g & 0 & 0 \\
0 & 0 & f & 0 \\
0 & 0 & 0 & g \\
\end{array} \right),  
\qquad
D =    \left( \begin{array}{cccc} 
0 & 0 & \partial_+ & 0 \\
0 & 0 & 0 & \partial_- \\
\partial_- & 0 & 0 & 0 \\
0 & \partial_+ & 0 & 0 \\
\end{array} \right),  
$$
$$ 
J =    \left( \begin{array}{cccc} 
J & 0 & 0 & 0 \\
0 & J & 0 & 0 \\
0 & 0 & J & 0 \\
0 & 0 & 0 & J \\
\end{array} \right),  
\qquad
\check\rho=    \left( \begin{array}{cccc} 
0 & 1 & 0 & 0 \\
1 & 0 & 0 & 0 \\
0 & 0 & 0 & 1 \\
0 & 0 & 1 & 0 \\
\end{array} \right),  
\qquad
\gamma =    \left( \begin{array}{cccc} 
1 & 0 & 0 & 0 \\
0 & 1 & 0 & 0 \\
0 & 0 & -1 & 0 \\
0 & 0 & 0 & -1 \\
\end{array} \right).
\qquad
$$  
We skip a straightforward demonstration that a unitary transformation exchanging the second 
and the fourth component of the vector in the Hilbert space makes
$$(C^\infty(M)\oplus \C^\infty(M), H \otimes \C^2, D, \gamma, J, \check\rho)$$ 
unitarily equivalent to the even spectral triple with twisted reality 
condition constructed as the tensor product in Example~\ref{tensex}.

\section{Conclusions}

The {\em untwisting} procedure presented in this note demonstrates that there
exist deep connections between real twisted spectral triples equipped with 
some versions of the reality structures that satisfy modified first-order conditions, with
the usual spectral triples that have $\nu$-twisted (i.e.\ satisfying the twisted first-order 
condition \eqref{to1cr} as well as conditions \eqref{tc} and \eqref{reg}) real structures.  This demonstrates, for example, that the procedure 
of conformal twisting applied to the Dirac operator of a usual real spectral triple that satisfies
the first-order condition yields, depending on the chosen conformal factor, either spectral triples 
with $\nu$-twisted reality (discussed in details in \cite{bcds}) or real twisted spectral triples  
\cite{LaMa}; see Example~\ref{exlama}.  

Note that the  case of twisted spectral triples in \cite{CoMo}, with the Dirac operator that comes from a Dirac operator of a usual spectral triple
rescaled by a general invertible element of the algebra,  provides yet another type of structure discussed in 
Example~\ref{como.twist}, that is a {\em twisted} spectral triple with an {\em untwisted} first-order condition. 
This kind of twisting can be seen as an immediate consequence of Proposition~\ref{prop.exchange}, which exchanges
the properties of a triple for the algebra $A$ and its commutant $JAJ^{-1}$, or -- more to the point -- 
of Corollary~\ref{cor.comm}. 

We can summarise here three different kinds of twisted reality conditions obtained by the conformal twisting of 
a real spectral triple $(A, H, \pi, D, J)$  in the following table:\\ 

\vspace{1.5ex}

\begin{tabular}{|p{4.5cm} | p{4.5cm}  | p{4.5cm} |}
 \hline
	 $(A, H, \pi, k'Dk',J)$ & $(A, H, \pi, kk'Dkk',J)$ & $(A,H, \pi, kDk,J)$ \\ \hline
spectral triple with the $\nu$-twisted real structure and first-order condition
	&  \vtop{\hsize=4.5cm \noindent real $\rho$-twisted spectral triple}
	 &  \vtop{\hsize=4.5cm \noindent twisted\,spectral\,triple\,with real\,structure\,and\,untwisted first-order condition}  \\ \hline 
 $\nu = k^{-1} k' $ &  $\rho = \mathrm{Ad}_{u^2}$ &  $\nu = k k'^{-1}$ \\ \hline
\end{tabular}
\vspace{1ex} \ \\
Here $k = \pi(u) \in \pi(A)$, where $u\in A$ is invertible and such that $k$ is positive with bounded inverse, $k' = JkJ^{-1}$ and we have 
$\nu J D = \epsilon' J D \nu$, and  $\nu J \nu = J$ in the first and the third cases.

The untwisting procedure described in Theorem~\ref{thm.main} leads from the middle column of the above table 
to the left one as shown in Example \ref{kkprime}, whereas the Corollary \ref{cor.comm} demonstrates that the
left and right column are dual in the sense that if one holds for an algebra $A$ then the other holds for 
$A^{op}$ with the same Dirac operator.

It should be however noted that the untwisting procedure has a much bigger scope than  conformal twists. 
The particular example of the minimal twist yields another interesting structure allowing one to untwist the 
real twisted spectral triple of an algebra that which is built over an algebra of an  even spectral triple with 
the grading included.  It might be worth mentioning at this point that, following \cite{DFLM},  the square root 
$\nu$ \eqref{numin} of the minimal twisting automorphism required for the untwisting  procedure can be interpreted as 
the Wick rotation (see \cite{CoMa,DLM}) in the noncommutative approach to the Standard Model \cite{CoMo}.

Finally, although the  twisted spectral triple obtained through a minimal twisting is not Lipschitz regular in the sense of  \cite{CoMo} and therefore it does not naturally represent an unbounded ,,twisted'' version of a Fredholm module, the untwisting procedure, together with the doubling presented at the end of Section~\ref{sec.main} allow one to link this construction to 
the usual triple and $K$-homology in terms of the Fredholm module. 

As the presented construction of {\em untwisting} transforms real twisted spectral triples to the usual spectral triples with twisted real structure, 
it allows then one to use the analytical tools developed for the spectral triples. Yet, not always the twist automorphism in
real twisted spectral triples in the sense of Definition~\ref{def:realtwist} is implemented externally. Such model case of a real twisted spectral triple, 
if exists,  cannot be transformed back to the spectral triple and will be a very interesting example of a genuine twisted 
geometry. 

Our study demonstrates also possible paths towards further generalisations of spectral triples with a real structure. 
In particular for a given algebra automorphism $\rho$ 
and $\nu\in B(H)$ with bounded inverse one could investigate $\rho$-twisted spectral triple $(A,H,\pi, D)$ that
has a {\em twisted real structure} $J$ of type $(\nu,\rho)$.
This requires that $J\nu $ preserves the domain of $D$ and 
$ DJ\nu = \pm \nu JD$, 
that $\nu J\nu = J$ and that the following {\em $(\nu,\rho)$-twisted 
first-order condition} holds:
\begin{equation}\label{t1oc.gen}
[[D,a]^\pi_\rho\, , b]^{\pi_J}_{\rho\circ {\hat\nu}^{-2}} =0, \qquad \forall a,b\in A.
\end{equation}
Then the spectral triple with $\nu$-twisted real structure corresponds to $\rho=\hbox{id}$, so it is of the type $(\hbox{id},\nu)$,
the case of real twisted spectral triples \cite{LaMa} can be identified with the type $(\rho,1)$ whereas the case of a 
twisted spectral triple in \cite{CoMo} corresponds to $\rho$ implemented by $\nu^2$ and its real structure is of type $(\hat{\nu}^2, \nu)$. 
We leave this exciting possibility for further studies.

\section*{Acknowledgements}
The authors would like to thank for hospitality the Institute of Mathematics of Polish Academy of Sciences (IMPAN), where the work on the present note started. Likewise, the first two authors are grateful for the hospitality of the Faculty of Physics, Astronomy and Applied Computer Science of the Jagiellonian University in Krak\'ow, where the work was completed. 
The research of all authors is partially supported by the Polish National Science Centre grant 2016/21/B/ST1/02438.

\end{document}